\documentclass[12pt,amstex,amssymb]{amsart}
\usepackage{graphicx}
\usepackage{tikz}
\usepackage{framed}
\usepackage{amsmath}
\usepackage{amssymb}

\newtheorem{prop}{Proposition}[section]
\newtheorem{theorem}[prop]{Theorem}
\newtheorem{rem}[prop]{Remark}
\newtheorem{lemma}[prop]{Lemma}
\newtheorem{cor}[prop]{Corollary}

\theoremstyle{definition}
\newtheorem{defi}[prop]{Definition}
\newtheorem{ex}[prop]{Example}

\def\cH{\mathcal{H}}

\def\cB{\mathcal{B}}

\def\cD{\mathcal{D}}

\def\cF{\mathcal{F}}

\def\cK{\mathcal{K}}

\def\cO{\mathcal{O}}
\def\cU{\mathcal{U}}
\def\R{\mathbb R}
\def\Z{\mathbb Z}
\def\C{\mathbb C}

\def\ux{\underline{x}}
\def\uy{\underline{y}}
\def\uz{\underline{z}}
\def\u0{\underline{0}}
\def\uomega{\underline{\omega}}

\def\b0{\mbox{\boldmath $0$}}

\def\gl2{{\rm GL}_2(\R)}
\def\gln{{\rm GL}_n(\R)}

\title{A Two Plus One Dimensional Continuous Wavelet Transform}

\author{Raja Milad}
\address{Department of Mathematics \& Statistics, Dalhousie University}
\email{raja.ma.milad@dal.ca}

\author{Keith F. Taylor}
\address{Department of Mathematics \& Statistics, Dalhousie University}
\email{keith.taylor@dal.ca}

\begin{document}

\begin{abstract}
The group $G_2$ of invertible affine transformations 
of $\R^2$ has, up to equivalence, one square--integrable
representation. Two new realizations of this
representation are presented and novel continuous
wavelet transforms, acting on functions of two plus
one variables, are derived.
\end{abstract}

\maketitle

\section{Introduction}
An analog of the continuous wavelet transform is 
introduced for the analysis of functions of two plus
one variables. This new transform
is derived from a square--integrable representation
of $G_2$, the group of all invertible affine transformations of the two dimensional plane. This is 
similar to how the continuous wavelet transform in
one variable can viewed as arising from a 
square--integrable representation of the group of affine transformations of the real 
line (see \cite{GMPI} and \cite{GMP}). We begin by
informally describing the transform and reconstruction
procedure. Any unexplained notation can be found in
Section 2.

We use column vectors $\ux=\begin{pmatrix}
x_1\\
x_2
\end{pmatrix}$ as two dimensional spacial variables
and row vectors $\uomega=(\omega_1,\omega_2)$ for 2
dimensional
frequency variables. We also use $t$ for a real
variable. Suppose $(\uomega,t)\to w(\uomega,t)$ is
a $\rho_2$--wavelet, a concept defined in Section 5.
For $\ux=\begin{pmatrix}
x_1\\
x_2
\end{pmatrix}$ and 
$A=\begin{pmatrix}
a & b\\
c & d
\end{pmatrix}$, a nonsingular real matrix, let
\begin{equation}\label{w_xA}
w_{\ux,A}(\uomega,t)=
\textstyle\frac{\|\uomega A\|}{\|\uomega\|}
e^{2\pi i\uomega\ux}\,w\left(\uomega A,
\frac{t-u_{\uomega,A}}{v_{\uomega,A}}\right),
\end{equation}
 where $u_{\uomega,A}$
 and $v_{\uomega,A}$ are rational functions of $\ux$ and $A$ as given in Equation 
\eqref{uv} in Section 4. The $\rho_2$--wavelet
transform is $f\to V_wf$, where 
\[
V_wf[\ux,A]=
\int_{\infty}^\infty\int_{\infty}^\infty
\int_{\infty}^\infty f(\omega_1,\omega_2,t)
\overline{w_{\ux,A}(\omega_1,\omega_2,t)}\,d\omega_1
d\omega_2dt.
\]
Then appropriate functions, $f$, can be recovered
from the $V_wf[\ux,A]$'s as follows:
\begin{equation}\label{intro_reconstruct}
f=\int_{\gl2}\int_{\R^2}V_wf[\ux,A]\,w_{\ux,A}\,
\frac{d\ux\,d\mu_{\gl2}(A)}{|\det(A)|}.
\end{equation}
The meaning of the vector--valued integral in
\eqref{intro_reconstruct} is in a weak sense. The
precise statement is given in 
\ref{reconstruction_final2}. The 
measure $\mu_{\gl2}$ is the Haar measure of
$\gl2$, the group of nonsingular $2\times 2$ real
matrices (see Equation \eqref{Haar_GL2}). 

An inspection of the manner in which $w$ is 
transformed in \eqref{w_xA} by the six real parameters incorporated
in the pair $(\ux,A)$ would indicate why we refer
to the $\rho_2$--wavelet transform as a two plus one
dimensional continuous wavelet transform. We view
the $\uomega$'s appearing in \eqref{w_xA} as two dimensional 
Fourier variables dual to spacial variables and the
$t$ as a time variable. In \eqref{w_xA}, functions
are transformed in the $t$ variable in a twisted
manner analogous to that done in the one dimensional
continuous wavelet transform. One the other hand,
action in the $\uomega$ variable is exactly the
Fourier transform version of the natural representation
of $G_2$ on functions of two variables, where all
$2\times 2$ nonsingular matrices as well as translations are available to 
manipulate a wavelet.
The main
purpose of this paper is to derive the $\rho_2$--wavelet
transform from an explicit realization of the unique
square--integrable representation of the group of all
invertible affine transformations of the plane. We
draw heavily on the results in a companion paper
\cite{MT}.

The theoretical basis for deriving a continuous
wavelet transform from a square--integrable
representation of a locally compact group is the Duflo--Moore Theorem (presented as Theorem \ref{Du_Mo} below)
which was established in \cite{DM}. An early case
where a two dimensional transform was associated 
with a square--integrable representation of a group of
affine transformations is \cite{Muz}.
In any dimension $n$, any affine transformation of
$\R^n$ is of the form $\uz\to A\uz+\ux$, for some
$n\times n$ real matrix $A$ and some $\ux\in\R^n$.
This transformation, denoted $[\ux,A]$,
will be invertible if and only
if $A$ is invertible; that is, $A\in\gln$. Let
$G_n=\{[\ux,A]:\ux\in\R^n, A\in\gln\}$. This is
a locally compact group when equipped with
composition of transformations as group product and 
the natural topology. When $H$ is a closed subgroup
of $\gln$, the set $\{[\ux,A]:\ux\in\R^n,A\in H\}$,
which we denote as $\R^n\rtimes H$, is a closed
subgroup of $G_n$. In the case dealt with in \cite{Muz},
$H=\left\{\begin{pmatrix}
r\cos(\theta) & -r\sin(\theta)\\
r\sin(\theta) & r\cos(\theta)
\end{pmatrix}: r>0, 0\leq\theta<2\pi\right\}$, where 
the natural representation of $\R^2\rtimes H$ as
defined in Example \ref{natural_rep} is square--integrable and yields a useful continuous wavelet 
transform (see \cite{ACMP} and \cite{AM}, for example). 
In \cite{BT}, it was recognized that when $H$ is a
closed subgroup of $\gln$ with the feature that
there are open and free $H$--orbits in $\widehat{\R^n}$
then the natural representation of $\R^n\rtimes H$
has square--integrable subrepresentations and, thus,
yields a continuous wavelet transform via the
Duflo--Moore Theorem. An orbit is {\em free} if 
each non--identity element of $H$ moves any point in
the orbit. For a comprehensive treatment
of this and related phenomena see \cite{Fuh}. The
continuous shearlet transform in two dimensions as 
applied in \cite{KuLa} and \cite{GLL} is another
useful transform that arises from a closed subgroup $H$
of $\gl2$ satisfying the assumption of \cite{BT} that
there exists an open and free $H$--orbit.

If $H=\gl2$, there is an open $H$--orbit in 
$\widehat{\R^2}$, but it is not a free orbit.
Nevertheless, $G_2=\R^2\rtimes\gl2$ 
has a square--integrable representation
(see Theorem \ref{List_2_d_2} below or Theorem 4.7 of
\cite{MT}). Much of \cite{MT} is devoted to deriving
and explicit expression for this representation. In
Section 4 below, we develop a unitarily equivalent
realization, given as Equation \eqref{sigma_1_1}, on
a more concrete Hilbert space. In Section 5, we 
convert the underlying Hilbert space again to arrive
at a realization of the representation that yields
what we call the $\rho_2$-continuous wavelet transform
as hinted at in \eqref{w_xA} and \eqref{intro_reconstruct} above.

In Section 2, we establish the notational conventions
we use and recall the necessary basic theory. The
theory of square--integrable representations is reviewed
in Section 3, including the Duflo--Moore Theorem,
the Duflo--Moore operator, and the derivation of a 
continuous wavelet transform from a square--integrable 
representation. We also establish Theorem \ref{List_2_d_2} which adds $\gl2$, itself, 
to the known list of
closed subgroups $H$ of $\gl2$ with the property that
$\R^2\rtimes H$ has a square--integrable
representation. The proof uses a key result from
\cite{KL}. We note that an irreducible representation
$\sigma$
of $\R^2\rtimes\gl2$, that is $G_2$, is constructed,
and a proof that $\sigma$ is square--integrable
is given, in \cite{MT}. The proof in \cite{MT} is
does not use the results of \cite{KL}. The formula
for $\sigma$ is given in Equation \eqref{sigma_pointwise_formula_1} in Section 4.

Also, in Section 4, we define a unitary map from the
Hilbert space of $\sigma$ onto a Hilbert space we
denote as 
$L^2\big(\widehat{\R^2}\times\widehat{\R}\big)$ to
emphasize the role the third component variable
plays. We denote by $\sigma_2$ the representation
obtained by moving $\sigma$ to this concrete Hilbert 
space. We write $\widehat{\R^2}$ and $\widehat{\R}$
because we view $\sigma_2$ as acting on the Fourier 
transform side. The formula for $\sigma_2$ is given
in Equation \eqref{sigma_1_1}. The main theorem in
Section 4 is Theorem \ref{Duflo_Moore_operator} where
the important Duflo--Moore operator associated with
$\sigma_2$ is identified. This allows the 
definition of $\sigma_2$--wavelets in Section 5
and an explicit statement of the $\sigma_2$--wavelet
transform and reconstruction formula in Theorem
\ref{reconstruction_final1}. We note that a method 
for determining the Duflo--Moore operator under
hypotheses that would apply to $G_2$ is derived
in \cite{ACDL}, but no actual examples are worked
out with their method. In the present paper, the
nature of the Duflo--Moore operator is suggested 
by our explicit proof that $\sigma_2$ is 
square--integrable and the uniqueness part of the
Duflo--Moore Theorem establishes the exact form of
the Duflo--Moore operator.
Finally, in Section 5, we execute the inverse
Fourier transform in the third variable to arrive
at the representation $\rho_2$ which provides the
transformations given in \eqref{w_xA} above.

\section{Preliminaries}
We use the following notational conventions: 
\begin{itemize}
\item $\R$ denotes the field of real numbers, often
viewed as a group under addition.
\item $\R^*=\R\setminus\{0\}$ is a group under
multiplication.
\item $\R^+=\{a\in\R^*:a>0\}$.
\item $\R^n=\left\{\ux=\begin{pmatrix}
x_1\\
\vdots\\
x_n
\end{pmatrix}:
x_1,\cdots,x_n\in\R\right\}$.
\item $\gln$ denotes the group of nonsingular 
$n\times n$ real matrices.
\item $\widehat{\R^n}=\{\uomega=
(\omega_1,\cdots\omega_n):
\omega_1,\cdots,\omega_n\in\R\}$.
\item $\widehat{\R^2}\times\widehat{\R}=
\{(\uomega,\omega_3):\uomega\in\widehat{\R^2},
\omega_3\in\widehat{\R}\}$, where $\widehat{\R}$
is identified with $\R$.
\item If 
$1\leq p\leq\infty$, $L^p(\R^n)$,
$L^p\big(\widehat{\R^n}\big)$, or
$L^p\big(\widehat{\R^2}\times\widehat{\R}\big)$ denote the
usual Lebesgue spaces (see, for example, \cite{Ru}).
\item The Fourier transform:
$\widehat{f}(\uomega)=\int_{\R^2}f(\ux)
e^{2\pi i\uomega\ux}d\ux$, for all $\uomega
\in\widehat{\R^n}$, for $f\in L^1(\R^n)$.
\item $\cF:L^2(\R^n)\to L^2\big(\widehat{\R^n}\big)$
is the unitary map satisfying 
$\cF f=\widehat{f}$, for $f\in L^1(\R^n)\cap 
L^2(\R^n)$.
\item For $\ux\in\R^n$ and $A\in\gln$, let 
$[\ux,A]$ denote the affine transformation of
$\R^n$ given by $[\ux,A]\uz=\ux+A\uz$, for all
$\uz\in\R^n$.
\end{itemize}
Let $G_n=\{[\ux,A]:\ux\in\R^n, A\in\gln\}$. Equipped with composition of affine transformations, $G_n$
is a group. The product of $[\ux,A]$ and $[\uy,B]$ in $G_n$ is
$[\ux,A][\uy,B]=[\ux+A\uy,AB]$. The identity is
$[\u0,{\rm id}]$, where ${\rm id}$ is the identity
matrix, and the inverse is given by 
$[\ux,A]^{-1}=[-A^{-1}\ux,A^{-1}]$. Both $\gln$
and $G_n$ are locally compact groups when endowed
with the obvious topologies. All the properties
of locally compact groups that we need can be
found in \cite{Fol}, \cite{KT}, or \cite{HR}. We
note that $G_n$ can be viewed as the
semidirect product $\R^n\rtimes\gln$ with
$\gln$ acting on $\R^n$ through matrix product. 
For any 
locally compact group, there always exists a nonzero
Radon measure that is invariant under 
left translations,
a {\em left Haar measure}. Such a measure is
unique up to a positive multiple. 
Let $C_c(X)$ denote 
the space of continuous
complex-valued functions of compact support on a
locally compact Hausdorff space $X$. 
 If $G$ is
a locally compact group, we can uniquely specify 
a left Haar measure on $G$ by showing how to
compute $\int_G f\,d\mu_G$, for any $f\in C_c(G)$.
For example,
\begin{equation}\label{Haar_GL2}
\int_{\gl2}f\,d\mu_{\gl2}=
\int_{-\infty}^\infty\int_{-\infty}^\infty
\int_{-\infty}^\infty\int_{-\infty}^\infty
f\begin{pmatrix}
w & x\\
y & z
\end{pmatrix}\frac{dw\,dx\,dy\,dz}{(wz-xy)^2},
\end{equation}
for all $f\in C_c(\gl2)$, where the integrals
in the right hand side of \eqref{Haar_GL2} are
just the Riemann integral (see Example 4 following 
Proposition 2.21 of \cite{Fol}).
It will be convenient
to sometimes use $\int_{\gl2}f(A)\,dA$ to denote
$\int_{\gl2}f\,d\mu_{\gl2}$. Left invariance means 
that $\int_{\gl2}f(BA)\,dA=\int_{\gl2}f(A)\,dA$,
for any $B\in\gl2$ and any $f$ for which the integral
makes sense. In fact, the integral in \eqref{Haar_GL2} is also right invariant. That is,
$\int_{\gl2}f(AB)\,dA=\int_{\gl2}f(A)\,dA$. Groups
with this property are called {\em unimodular}.
However, the group $G_2$ is nonunimodular.

For any locally compact group $G$, there exists
a continuous homorphism $\Delta_G:G\to\R^+$, 
called the {\em modular function} of $G$ such
that, for any $y\in G$, 
\[
\int_Gf(x)\,d\mu_G(x)=\Delta_G(y)
\int_Gf(xy)\,d\mu_G(x),
\]
 for any $f$ for which
the integrals make sense.
You can directly verify with standard
change of variables that, for any $[\uy,B]
\in G_2$,
\[
\int_{\gl2}\int_{\R^2}
f\big([\uy,B][\ux,A]\big)
\frac{d\ux\,dA}{|\det(A)|}=
\int_{\gl2}\int_{\R^2}
f[\ux,A]\frac{d\ux\,dA}{|\det(A)|}
\]
and
\[
|\det(B)|^{-1}\int_{\gl2}\int_{\R^2}
f\big([\ux,A][\uy,B]\big)
\frac{d\ux\,dA}{|\det(A)|}=
\int_{\gl2}\int_{\R^2}
f[\ux,A]\frac{d\ux\,dA}{|\det(A)|},
\]
for any $f\in C_c(G_2)$.
Thus, the left Haar integral on $G_2$ is given by
\begin{equation}\label{Haar_G_2}
\int_{G_2}f\,d\mu_{G_2}=
\int_{\gl2}\int_{\R^2}
f[\ux,A]\frac{d\ux\,dA}{|\det(A)|}, 
\text{ for all } f\in C_c(G_2)
\end{equation}
and $\Delta_{G_2}[\uy,B]=|\det(B)|^{-1}$, for
all $[\uy,B]\in G_2$.

Let $\cH$ be a Hilbert space with inner product
$\langle\cdot,\cdot\rangle_{\cH}$ and
norm $\|\cdot\|_{\cH}$. Let $\cB(\cH)$ denote
the Banach $*$-algebra of bounded linear operators
on $\cH$. The {\em weak operator topology} on
$\cB(\cH)$ is the weakest topology such that
$T\to\langle T\xi,\eta\rangle_{\cH}$ is a continuous function on
$\cB(\cH)$,
for all $\xi,\eta\in\cH$.

If $\cK$ is another Hilbert space, a
linear map $U:\cH\to\cK$ is a {\em unitary map}
if it is a surjective isometry; that is, if
$U\cH=\cK$ and
$\|U\xi\|_{\cK}=\|\xi\|_{\cH}$, for all $\xi\in\cH$.
The polarization identity implies that $U$ then
preserves inner products. Let $\cU(\cH)$ denote
the set of all unitary maps from $\cH$ to $\cH$.
Under composition of operators, $\cU(\cH)$ is
a group with identity $I$, where $I\xi=\xi$, for
all $\xi\in\cH$.
All that we need about
Hilbert spaces and operators on and between them can
be found in \cite{Ru} or \cite{KR}.
If $(X,\Sigma,\mu)$ is a measure space, then
$L^2(X,\mu)$ is a Hilbert space (Example 4.5 of
\cite{Ru}). If $G$ is a locally compact group,
$L^2(G,\mu_G)$ will simply be denoted $L^2(G)$.

An unbounded linear operator plays a critical
role in this paper. See Section 2.7 of \cite{KR}
for a basic discussion of unbounded operators on
a Hilbert space. Following Example 2.7.2 of 
\cite{KR}, let $(X,\Sigma,\mu)$ be a measure space 
and let $g:X\to\C$ be a measurable function. Let
$\cD_{M_g}=
\left\{f\in L^2(X,\mu):\int_X|gf|^2\,d\mu<\infty
\right\}$ and define $M_g:\cD_{M_g}\to L^2(X,\mu)$
by $M_gf=gf$, for all $f\in\cD_{M_g}$. Then 
$\cD_{M_g}$ will be a dense linear subspace of 
$L^2(X,\mu)$. 
Note that $\cD_{M_g}=L^2(X,\mu)$ if
and only if $g$ is essentially bounded; $M_g
\in\cB(\cH)$ in this case. Even when $g$ is not essentially 
bounded, $M_g$ is a closed operator. If $g$ is
real-valued, then $M_g$ is selfadjoint.
If $g(x)\geq 0$, for all $x\in X$,
then $M_g$ is a positive operator.

Let $G$ be a locally compact group. 
A {\em 
unitary representation} $\pi$ of $G$ on a Hilbert space $\cH_\pi$ is a homomorphism 
 $\pi:G\to\cU(\cH_\pi)$
that is continuous when $\cU(\cH_\pi)$ is equipped 
with the weak operator topology. We will often just write {\em representation} to mean unitary 
representation. The {\em left
regular representation} of $G$ is denoted 
$\lambda_G$, which acts on $\cH_{\lambda_G}=L^2(G)$
by $\lambda_G(x)f(y)=f(x^{-1}y)$, for $\mu_G$-a.e.
$y\in G$, for $f\in L^2(G)$ and each $x\in G$.
If $\pi$ and $\sigma$ are two representations of 
$G$ and there exists a unitary map $U:\cH_\pi\to
\cH_\sigma$ such that $U\pi(x)=\sigma(x)U$, for
all $x\in G$, then we say $\pi$ and $\sigma$ are
equivalent.

Fix a
representation $\pi$ of $G$.
For each $\eta,\xi\in\cH_\pi$, define a complex-valued
function $V_\eta\xi$ on $G$ by
\begin{equation}\label{voice}
V_\eta\xi(x)=\langle\xi,\pi(x)\eta\rangle_{\cH_\pi},
\quad\text{for all }x\in G.
\end{equation}
Let $C_b(G)$ denote
the Banach space of bounded continuous complex-valued
functions on $G$ equipped with the supremum norm.
The continuity assumption on $\pi$ and the
Cauchy-Schwarz inequality imply that $V_\eta\xi
\in C_b(G)$. If $\eta\in\cH_\pi$ is held fixed, then
the map $V_\eta:\xi\to V_\eta\xi$ is a continuous linear
map of $\cH_\pi$ into $C_b(G)$. 

A subspace $\cK$ of
$\cH_\pi$ is called {\em $\pi$-invariant} if
$\xi\in\cK$ implies $\pi(x)\xi\in\cK$, for all
$x\in G$. If $\cK$ is a closed subspace, then
$\cK$ is $\pi$-invariant if and only if $\cK^\perp$
is $\pi$-invariant. If $\cK$ is a $\pi$-invariant 
closed subspace of $\cH_\pi$, define $\pi_{\cK}(x)
=\pi(x)\big|_{\cK}$, for all $x\in G$. Then 
$\pi_{\cK}$ is a representation of $G$ on the
Hilbert space $\cK$. We call $\pi_{\cK}$ a 
{\em subrepresentation} of $\pi$ and we can view
$\pi$ as the direct sum of $\pi_{\cK}$ and
$\pi_{\cK^\perp}$ in the obvious manner.
We say that $\pi$ is 
{\em irreducible} when $\{0\}$ and $\cH_\pi$
are the only $\pi$-invariant closed subspaces
of $\cH_\pi$. A representation $\pi$ is irreducible
if and only if $V_\eta$ is injective for every
nonzero $\eta\in\cH_\pi$ (see Proposition 1.32, 
\cite{KT}).

\begin{ex}\label{natural_rep}
Let $H$ be a closed subgroup of $\gln$ and let
$\R^n\rtimes H =
\{[\ux,A]:\ux\in\R^n,A\in H\}$,
 a closed subgroup of
$G_n$. The action of $\R^n\rtimes H$ on $\R^n$
as affine transformations yields a representation
of $\R^n\rtimes H$ on the Hilbert space $L^2(\R^n)$ as follows. For
$[\ux,A]\in\R^n\rtimes H$ and $f\in L^2(\R^n)$,
\begin{equation}\label{natural}
\rho[\ux,A]f(\uy)=|\det(A)|^{-1/2}f\big(
A^{-1}(\uy-\ux)\big), \text{ for a.e. }\uy\in\R^n.
\end{equation}
It is easy to verify that each $\rho[\ux,A]$ is
 a homomorphism of $\R^n\rtimes H$ into
$\cU\big(L^2(\R^n)\big)$. See page 68 of \cite{Fol}
for a more general situation where the required
continuity is established.
We will call $\rho$ the {\em natural
representation} of $\R^n\rtimes H$.
Note that $\rho$ is also often called the quasi-regular
representation of $\R^n\rtimes H$. Let's use the
unitary map $\cF:L^2(\R^n)\to L^2\big(\widehat{\R^n}
\big)$ to create a new representation of 
$\R^n\rtimes H$. For each $[\ux,A]\in \R^n\rtimes H$,
let 
$\widehat{\rho}[\ux,A]=\cF\rho[\ux,A]\cF^{-1}$. 
A simple
calculation (see Proposition 1 of \cite{BT}, for
example) shows that
\begin{equation}\label{natural_Fourier}
\widehat{\rho}[\ux,A]\xi(\uomega)=|\det(A)|^{1/2}
e^{2\pi i\uomega\ux}\xi(\uomega A), \text{ for a.e. }
\uomega\in\widehat{\R^n} \text{ and all }
\xi\in L^2\big(\widehat{\R^n}\big).
\end{equation}
Then $\widehat{\rho}$ is a 
representation of $\R^n\rtimes H$
that is equivalent to the natural representation.
Now, suppose $\cK$ is a closed 
$\widehat{\rho}$--invariant
subspace of $L^2\big(\widehat{\R^n}\big)$. Let 
$P_{\cK}$ be the orthogonal projection of
$L^2\big(\widehat{\R^n}\big)$ onto $\cK$. 
Since $\cK$ is
$\widehat{\rho}$-invariant, $P_{\cK}$ commutes with
$\widehat{\rho}[\ux,A]$, 
for all $[\ux,A]\in \R^n\rtimes H$.
For each $\ux\in\R^n$, 
$\widehat{\rho}[\ux,{\rm id}]\xi(\uomega)
=e^{2\pi i\uomega\ux}\xi(\uomega)$, for a.e. 
$\uomega\in\widehat{\R^n}$ and each $\xi\in 
L^2\big(\widehat{\R^n}\big)$. Since $P_{\cK}$
commutes with pointwise multiplication by all
the characters of $\R^n$, there exists a Borel
subset $E\subseteq\widehat{\R^n}$ such that
$P_{\cK}\xi={\bf 1}_E\xi$, for all $\xi\in 
L^2\big(\widehat{\R^n}\big)$ (Theorem 9.17, 
\cite{Ru}). Now, for each 
$A\in H$, 
$P_{\cK}=\widehat{\rho}[\u0,A]P_{\cK}\widehat{\rho}[\u0,A]^{-1}$
implies
\[
{\bf 1}_E(\uomega)\xi(\uomega)=P_{\cK}\xi(\uomega)=
\widehat{\rho}[\u0,A]P_{\cK}\widehat{\rho}[\u0,A]^{-1}\xi(\uomega)=
{\bf 1}_E(\uomega A)\xi(\uomega),
\]
for a.e. $\uomega\in\widehat{\R^n}$ and all
$\xi\in L^2\big(\widehat{\R^n}\big)$. This 
implies that the symmetric difference of 
$EA$ and $E$ is a null set for every $A\in H$. 
Conversely, if $E$ is a
Borel subset of $\widehat{\R^n}$ with the property
that $EA$ agrees with $E$ up to a set of measure 0,
then $\cK_E=\{{\bf 1}_E\xi:\xi\in 
L^2\big(\widehat{\R^n}\big)\}$ is a 
$\widehat{\rho}$--invariant closed subspace of 
$L^2\big(\widehat{\R^n}\big)$.
We are interested in the simple case that 
$E$ is a single $H$-orbit. That is, for some,
hence any, $\uomega\in E$, $\{\uomega A:A\in H\}=
E$. If $E$ is a single $H$--orbit of positive
Lebesgue measure, then the restriction of 
$\widehat{\rho}$
to $\cK_E$ is an irreducible subrepresentation
of $\widehat{\rho}$. The
easiest way for this to arise is if there are
$H$--orbits in $\widehat{\R^n}$ that are open
subsets of $\widehat{\R^n}$. For example, if
$H=\gln$, itself, then $\cO=\widehat{\R^n}\setminus
\{\u0\}$ is a $\gln$--orbit. Since $\cO$ is
co-null, $\cK_{\cO}=L^2\big(\widehat{\R^n}\big)$.
Therefore, the representation $\widehat{\rho}$ given by
\eqref{natural_Fourier} is an irreducible
representation of $G_n$. Since it is equivalent to
$\widehat{\rho}$, the natural representation $\rho$ of $G_n$
is irreducible, for any positive integer $n$.
\end{ex}
\begin{rem}
Notice the similarity between the action of $\ux$ and $A$
on the $\uomega$ variable in \eqref{natural_Fourier}
and in \eqref{w_xA}.
\end{rem}

\section{Square--Integrable Representations}
 There is an extra property that some
irreducible unitary representations have that is
critical to us. 
\begin{defi}
A unitary representation $\pi$ of a locally compact
group
$G$ is {\em square--integrable} if $\pi$ is
irreducible and there exists a nonzero 
$\eta\in\cH_\pi$
such that $V_\eta\eta\in L^2(G)$.
\end{defi}
Our primary reference for square-integrable 
representations is \cite{DM}. The following theorem
is formulated from Theorem 3 of \cite{DM} and 
the preliminary results in \cite{DM}. We organize this statement in a manner convenient for our
purposes. Also see Theorem 2.25 of \cite{Fuh}.
\begin{theorem}\label{Du_Mo} {\bf [Duflo-Moore]}
Let $\pi$ be a square-integrable representation
of a locally compact group. Let $\cD_\pi=\{\eta:
V_\eta\eta\in L^2(G)\}$. Then\\
\indent {\rm (a)} $\cD_\pi$ is a dense
linear subspace of $\cH_\pi$,\\
\indent {\rm (b)} if $\eta\in\cD_\pi$,
then $V_\eta\xi\in L^2(G)$, for all $\xi\in\cH_\pi$,
\\
\indent {\rm (c)} there exists a nonzero
selfadjoint positive operator $C_\pi$ in $\cH_\pi$
with domain $\cD_\pi$ satisfying
\[
\pi(x)C_\pi\pi(x)^*=\Delta_G(x)^{1/2}C_\pi, 
\text{ for all } x\in G,
\]
\indent {\rm (d)} for any $\xi_1,\xi_2\in\cH_\pi$
and $\eta_1,\eta_2\in\cD_\pi$,
\[
\langle V_{\eta_1}\xi_1,
V_{\eta_2}\xi_2\rangle_{L^2(G)}=
\langle\xi_1,\xi_2\rangle_{\cH_\pi}
\langle C_\pi\eta_2,C_\pi\eta_1\rangle_{\cH_\pi}
\]
Moreover, if $T$ is any densely defined nonzero
selfadjoint positive operator in $\cH_\pi$ that
satisfies the identity $\pi(x)T\pi(x)^*=\Delta_G(x)^{1/2}T$,
for all $x\in G$, then the domain of $T$ is $\cD_\pi$ 
and there exists a constant 
$r> 0$ such that $T=rC_\pi$.
\end{theorem} 

When $\pi$ is a square-integrable representation
of $G$, the elements of $\cD_\pi$ are called 
{\em admissible vectors}. If $\eta\in\cD_\pi$
satisfies $\|C_\pi\eta\|_{\cH_\pi}=1$, then
$\eta$ is called a {\em $\pi$-wavelet}. When
$\eta$ is a $\pi$-wavelet and $\eta_1=\eta_2=\eta$,
then Theorem \ref{Du_Mo} (d) says, for any
$\xi_1,\xi_2\in\cH_\pi$,
\begin{equation}\label{V_eta_inner_prod}
\langle V_{\eta}\xi_1,
V_{\eta}\xi_2\rangle_{L^2(G)}=
\langle\xi_1,\xi_2\rangle_{\cH_\pi}.
\end{equation}
In particular, $V_\eta$ is a linear isometry of 
$\cH_\pi$ into $L^2(G)$. As is well-known (see for 
example \cite{GMP}, \cite{BT} or \cite{Fuh}), \eqref{V_eta_inner_prod} implies the
{\em reconstruction formula}. If $\eta$ is a 
$\pi$-wavelet, then, for any $\xi\in\cH_\pi$,
\begin{equation}\label{reconstruction}
\xi=\int_G V_\eta\xi(x)\pi(x)\eta\,d\mu_G(x),
\quad\text{weakly in }\cH_\pi.
\end{equation} 
The classical continuous wavelet transform on $\R$
arises in this manner from a square-integrable
representation of $G_1$ the group of all invertible
affine transformations of the real line (for example,
see \cite{GMP} and \cite{GMPII}). Left Haar integration on $G_1$ is given by
\[
\int_{G_1}f\,d\mu_{G_1}=
\int_{-\infty}^\infty\int_{-\infty}^\infty
f[x,a]\,\frac{dx\,da}{a^2}, \quad\text{for }
f\in C_c(G_1).
\]
Recall from Example \ref{natural_rep} that the natural representation of $G_1$ is given by, for
$[x,a]\in G_1$,
$\rho[x,a]f(t)=
|a|^{-1/2}f\left(\frac{t-x}{a}\right)$, 
for a.e. $t\in\R$
and any $f\in L^2(\R)$, and it is irreducible. 
It is also square-integrable
(see Example 3, Section 11, of \cite{KL}). A clear
and direct calculation given in the proof of
Theorem 3.3.5 of \cite{HW}, adapted from 
\cite{GMPI}, simultaneously shows that $\rho$ is
irreducible and square-integrable. Note that 
Theorem 3.3.5 \cite{HW} concerns the subgroup of
$G_1$ consisting of elements $[x,a]$ with $a>0$
parametrized by $(u,v)\to[v,e^u]$.
But a slightly modified calculation works 
for all of $G_1$
and yields the key identity:
\begin{equation}\label{G_1_identity}
\int_{-\infty}^\infty\int_{-\infty}^\infty|
\langle f,\rho[x,a]g\rangle|^2\,\frac{dx\,da}{a^2}
= \|f\|_2^{\,2}\int_{-\infty}^\infty 
\frac{|\widehat{g}(\omega)|^2}{|\omega|}\,d\omega,
\text{ for }f,g\in L^2(\R).
\end{equation}
So $V_g$ is injective as long as $g\neq 0$. This is 
another way of showing $\rho$ is irreducible. Moreover, $\rho$ is
square-integrable, because
$V_gf\in L^2(G_1)$, for all $f\in L^2(\R)$,
exactly when $\int_{-\infty}^\infty 
\frac{|\widehat{g}(\omega)|^2}{|\omega|}\,d\omega$
is finite and we see that 
$g\in L^2(\R)$ is a $\rho$-wavelet
when $\int_{-\infty}^\infty 
\frac{|\widehat{g}(\omega)|^2}{|\omega|}\,d\omega=1$.
This line of reasoning was generalized to
certain closed subgroups of higher dimensional affine
groups in \cite{BT} and \cite{Fuh2}. 

Let $H$ be a closed subgroup of $\gln$ and 
consider the natural representation of 
$\R^n\rtimes H$ on $L^2(\R^n)$ as in Example \ref{natural_rep}.  
The group $H$ acts on $\widehat{\R^n}$ on the right
by matrix multiplication. The {\em $H$--orbit} of
$\uomega\in\widehat{\R^n}$ is
$\uomega H=\{\uomega A:A\in H\}$
and the {\em stability subgroup} in $H$ of 
$\uomega$ is 
$H_{\uomega}=\{A\in H:\uomega A=\uomega\}$.
See Theorem 3.1 of 
\cite{CFO} and the references therein for the
following result.
\begin{theorem}\label{Fuhr_1}
Let $H$ be a closed subgroup of $\gln$ and let
$\rho$ be the representation of $\R^n\rtimes H$
given in \eqref{natural}. 
Then $\rho$ is a direct sum of 
square-integrable representations
if and only if there exists an
$\uomega\in\widehat{\R^n}$ such that $\uomega H$ is
open in $\widehat{\R^n}$ and 
$H_{\uomega}$ is compact.
\end{theorem}
If $H$ is such that $\rho$ is square-integrable
and if $H_0$ is an open subgroup of $H$, then 
the restriction of $\rho$ to $\R^2\rtimes H_0$
will be a direct sum of finitely many 
irreducible representations, each of which will 
be square-integrable. Classifications, up to 
conjugation, of all connected subgroups of $\gln$
with open orbits and compact stability subgroups
for points in the open orbits are given in
\cite{Fuh3} for $n=2$ and \cite{CFO} for $n=3$.
With a small amount of consideration, we can 
formulate the following list based on \cite{Fuh3}.
In all cases $a$ and $b$ are real.
\[
H^D=\left\{\begin{pmatrix}
a & 0\\
0 & b
\end{pmatrix}:ab\neq0\right\},
H^S=\left\{\begin{pmatrix}
a & b\\
-b & a
\end{pmatrix}:a^2+b^2>0\right\},
H^\alpha=\left\{\begin{pmatrix}
a & b\\
0 & |a|^\alpha
\end{pmatrix}:a\neq0\right\},
\]
where $\alpha\in\R$. Note that we have modified
the list given in \cite{Fuh3} or \cite{CFO} to
identify those $H$ with a co-null $H$--orbit in
$\widehat{\R^2}$. Each of these just has a small
number of open subgroups that are easy to 
identify.
\begin{prop}\label{List_2_d}
Let $H$ be a closed subgroup of $\gl2$ and let $\rho$ be the natural representation of $\R^2\rtimes H$
as given in \eqref{natural}. Then $\rho$ is a 
square-integrable representation
if and only if $H$ is conjugate to 
one of $H^D$, $H^S$, or $H^\alpha$, for some $\alpha\in\R$. Moreover, $\rho$ is a direct sum of
square-integrable representations if and only if
$H$ is conjugate to an open subgroup of some
member of this list.
\end{prop}
Each of the {\em diagonal group} $H^D$, the 
{\em similitude group} $H^S$, and the 
{\em shearlet groups} $H^\alpha$ and the
associated continuous wavelet transforms, or discrete
versions, are useful in image processing and
analysis of data that is two dimensional in some sense. Proposition \ref{List_2_d} provides the
answer to the question of which closed subgroups
$H$ of $\gl2$ are such that the natural
representation of $\R^2\rtimes H$ on $L^2(\R^2)$
is square-integrable. However, if instead we ask 
which closed subgroups $H$ of $\gl2$ are such that
$\R^2\rtimes H$ has a square-integrable representation, we can use a result from \cite{KL}
to enlarge the list. It is not a perfect answer
because the tools used in \cite{KL} depend on
$H$ acting on $\widehat{\R^2}$ in a well-behaved manner.
Recall that a subset $Y$ of a topological space 
$X$ is {\em locally closed} if $Y$ is an open
subset of $\overline{Y}$.
\begin{theorem}\label{List_2_d_2}
Let $H$ be a closed subgroup of $\gl2$ with the
property that each $H$-orbit in $\widehat{\R^2}$
is locally closed. There 
exists a square-integrable representation of 
$\R^2\rtimes H$ if and only if $H$ is an open
subgroup of one of $H^D$, $H^S$, $H^\alpha$, for 
some $\alpha\in\R$, or $\gl2$ itself.
\end{theorem}
\begin{proof}
The assumption that the $H$-orbits are locally closed
in $\widehat{\R^2}$ means that Mackey's theory of
induced representations, as developed in \cite{Mac}
can be used to describe all the irreducible
representations of $\R^2\rtimes H$. In particular,
the hypotheses of Theorem 4.29 of \cite{KT} are
satisfied. Thus, for any irreducible representation
$\pi$ of $\R^2\rtimes H$, there exists an $\uomega
\in\widehat{\R^2}$ and an irreducible representation
$\psi$ of the stabilizer $H_{\uomega}$ such that
$\pi$ is equivalent to the induced representation
${\rm ind}_{\R^2\rtimes 
H_{\uomega}}^{\R^2\rtimes H}(\chi_{\uomega}\otimes
\psi)$. The representation $\chi_{\uomega}\otimes
\psi$ of $\R^2\rtimes H_{\uomega}$ is given by
$(\chi_{\uomega}\otimes\psi)[\ux,A]=
e^{2\pi i\uomega\ux}\psi(A)$, for all 
$[\ux,A]\in \R^2\rtimes H_{\uomega}$. For the definition and basic properties of induced representations, one can consult \cite{Fol} or \cite{KT}, but this is not essential for anything
in this paper beyond the current section. By
Corollary 11.1 of \cite{KL}, ${\rm ind}_{\R^2\rtimes 
H_{\uomega}}^{\R^2\rtimes H}(\chi_{\uomega}\otimes
\psi)$ is square-integrable if and only if 
$\psi$ is a square-integrable representation of
$H_{\uomega}$ and the orbit $\uomega H$ has
positive Lebesgue measure in $\widehat{\R^2}$.
Since $H_{\uomega}$ is a closed subgroup of 
$\gl2$ and the orbit $\uomega H$ is locally closed,
we have that $\uomega H$ is a sub-manifold of 
$\widehat{\R^2}$. Thus, either $\uomega H$ is
open or it is a null set. Therefore, if 
$\R^2\rtimes H$ has a square--integrable 
representation, then there exists an $\uomega
\in\widehat{\R^2}$ such that $\uomega H$ is open
and $H_{\uomega}$ has a square--integrable
representation. We now consider the dimension, 
$\dim(H)$, of
$H$ as a differentiable manifold.

Suppose $\uomega\in\widehat{\R^2}$ is such that
$\uomega H$ is open in $\widehat{\R^2}$. The map
$AH_{\uomega}\to \uomega A^{-1}$ is a diffeomorphism
of $H/H_{\uomega}$ with $\uomega H$. Thus,
$\dim(H_{\uomega})=\dim(H)-2$. If $\dim(H)$ is
2 or 3, then $H_{\uomega}$ is either discrete 
or one dimensional. In either case, if $H_{\uomega}$
has a square--integrable representation, it is compact. By Theorem \ref{Fuhr_1} and Proposition
\ref{List_2_d}, when $\dim{H}<4$ and $\R^2\rtimes H$
has a square--integrable representation, $H$ is
conjugate to an open subgroup of one of $H^D$, 
$H^S$, or $H^\alpha$, for some $\alpha\in\R$.

Finally, if $\dim{H}=4$, then $H$ is an open subgroup
of $\gl2$. So, either $H=\gl2$ or $H=\gl2^+=
\{A\in\gl2:\det(A)>0\}$. We will deal with 
$H=\gl2$, then mention the variation for $\gl2^+$. There are
just two $\gl2$-orbits in $\widehat{\R^2}$, $\{(0,0)\}$
and $\cO=\widehat{\R^2}\setminus\{(0,0)\}$. It is
convenient to use $\uomega=(1,0)$ as a representative
point in $\cO$. For $A=\begin{pmatrix}
a & b\\
c & d
\end{pmatrix}\in H=\gl2$, $(1,0)A=(a,b)$, so $A\in 
H_{(1,0)}$ if and only $a=1$ and $b=0$. Thus,
 $H_{(1,0)}=\left\{\begin{pmatrix}
1 & 0\\
u & v
\end{pmatrix}:u\in\R,v\in\R^*\right\}$. The map
$\begin{pmatrix}
1 & 0\\
u & v
\end{pmatrix}\to[u,v]$ is easily seen to be a
topological group isomorphism of $H_{(1,0)}$
with $G_1$. As we saw above, the natural
representation of $G_1$ is square--integrable.
Therefore, since $\cO$ is an open $H$-orbit, 
Corollary 11.1 of \cite{KL} implies there exists
a square--integrable representation of 
$\R^2\rtimes H=G_2$. If $H=\gl2^+$, then 
$H_{(1,0)}$ is isomorphic to $G_1^+=\{[b,a]:b\in\R,a\in\R^+\}$ and $G_1^+$ has two inequivalent
square--integrable representations (see \cite{BT},
for example). Thus, Corollary 11.1 of \cite{KL}
implies $\R^2\rtimes H$ has two 
square--integrable representations.
\end{proof}
\begin{rem}
There are closed subgroups of $\gl2$ which have 
some orbits that are not locally closed. For example,
there are ${\rm SL}_2(\Z)$-orbits that are dense in
$\widehat{\R^2}$ (see \cite{GW}). Such an orbit 
is not locally closed. The
question of whether or not 
$\R^2\rtimes{\rm SL}_2(\Z)$ has a square--integrable
representation will not be treated here.
\end{rem}

\section{A Particular Square-integrable Representation}

From the proof of Theorem \ref{List_2_d_2}, we
see that $G_2=\R^2\rtimes\gl2$ has a square--integrable representation. It arises as 
${\rm ind}_{\R^2\rtimes H_{(1,0)}}^{G_2}(\chi_{(1,0)}
\otimes\rho)$, where $\rho$ is the natural
representation of $G_1$ transferred to $H_{(1,0)}$.
Note that $\chi_{(1,0)}\otimes\rho$ is a representation
of $\R^2\rtimes H_{(1,0)}$ acting on $L^2(\R)$ by,
for $\left[\begin{pmatrix}
x_1\\
x_2
\end{pmatrix},\begin{pmatrix}
1 & 0\\
u & v
\end{pmatrix}\right]\in \R^2\rtimes H_{(1,0)}$ and
$f\in L^2(\R)$,
\[
(\chi_{(1,0)}\otimes\rho)\left[\begin{pmatrix}
x_1\\
x_2
\end{pmatrix},\begin{pmatrix}
1 & 0\\
u & v
\end{pmatrix}\right]f(t)=
|v|^{-1/2}e^{2\pi ix_1}f\left(\frac{t-u}{v}\right),
\text{ for a.e. } t\in\R.
\]
The inducing construction, as presented in 
Section 6.1 of \cite{Fol} or Section 2.3 of 
\cite{KT}, is designed to work in the most general 
of situations. As a result, the definition of 
the Hilbert space of the induced representation
is not transparent. Much of the work in 
\cite{MT} involves converting, via unitary maps, that abstract Hilbert space into a concrete 
$L^2$-space and tracking how 
${\rm ind}_{\R^2\rtimes H_{(1,0)}}^{G_2}(\chi_{(1,0)}
\otimes\rho)$ is transformed in that process. The
calculations in \cite{MT} make use of the fact that
the subgroup $H_{(1,0)}$ of $\gl2$ has a complementary
subgroup. Let $K_0=\left\{\begin{pmatrix}
s & -t\\
t & s
\end{pmatrix}:s,t\in\R,s^2+t^2>0\right\}$. Then
$K_0$ is a closed subgroup of $\gl2$. Note that the map
$\begin{pmatrix}
s & -t\\
t & s
\end{pmatrix}\to s+it$ is an isomorphism of $K_0$
with $\C^*=\C\setminus\{0\}$, the multiplicative group
of nonzero complex numbers. Then $K_0\cap H_{(1,0)}=
\{{\rm id}\}$ and $(M,C)\to MC$ is a homeomorphism
of $K_0\times H_{(1,0)}$.

This results in a concrete
square--integrable representation,
which we denote as $\sigma$, of $G_2$
acting\footnote{In \cite{MT}, the Hilbert space of
$\sigma$ is given as $L^2\big(K,L^2(\R^*)\big)$,
where $K=\{[\u0,L]:L\in K_0\}$. The change from $K$ to
$K_0$ is essentially notational.} on the Hilbert space 
$L^2\big(K_0,L^2(\R^*)\big)$, where $K_0$ and $\R^*$
are equipped with their respective Haar integrals. 
For $f\in C_c(\R^*)$, $\int_{\R^*}f\,d\mu_{\R^*}=
\int_{-\infty}^\infty f(\nu)\frac{d\nu}{|\nu|}$, while 
for $f\in C_c(K_0)$, we have
$\int_{K_0}f\,d\mu_{K_0}=
\int_{-\infty}^\infty \int_{-\infty}^\infty
f\begin{pmatrix}
s & -t\\
t & s
\end{pmatrix}
\frac{ds\,dt}{s^2+t^2}$.
The expression for $\sigma$ is greatly
simplified by defining two 
rational functions of six variables. 
For each $\uomega=(\omega_1,\omega_2)
\in\cO=
\widehat{\R^2}\setminus\{(0,0)\}$ and 
$A=\begin{pmatrix}
a & b\\
c & d
\end{pmatrix}\in\gl2$, let
\begin{equation}\label{uv}
\begin{split}
u_{\uomega,A} & 
=\frac{(ac+bd)(\omega_1^2-\omega_2^2)-(a^2+b^2-c^2-d^2)\omega_1\omega_2}{(a\omega_1+c\omega_2)^2+(b\omega_1+d\omega_2)^2} \text{ and}\\
v_{\uomega,A} & 
=\frac{(ad-bc)(\omega_1^2+\omega_2^2)}{(a\omega_1+c\omega_2)^2+(b\omega_1+d\omega_2)^2}
=\frac{\det(A)\|\uomega\|^2}{\|\uomega A\|^2}.
\end{split}
\end{equation} 
These functions arose during calculations in \cite{Mil}
involving the group structure of $\gl2$ and 
the map $\gamma:\cO\to K_0$ given by 
$\gamma(\uomega)=\frac{1}{\|\uomega\|^2}
\begin{pmatrix}
\omega_1 & -\omega_2\\
\omega_2 & \omega_1
\end{pmatrix}$, for all $\uomega\in\cO$. What was useful about this map was that $\gamma(\uomega)^{-1}
=\begin{pmatrix}
\omega_1 & \omega_2\\
-\omega_2 & \omega_1
\end{pmatrix}$, so 
$(1,0)\gamma(\uomega)^{-1}=\uomega$ and 
$\uomega\gamma(\uomega)=(1,0)$. Thus, for any
$A\in\gl2$,
\begin{equation}\label{stabilize_1_0}
(1,0)\big(\gamma(\uomega)^{-1}A\gamma(\uomega A)\big)
=\uomega A\gamma(\uomega A)=(1,0).
\end{equation}
Recall that $H_{(1,0)}=\left\{\begin{pmatrix}
1 & 0\\
u & v
\end{pmatrix}:u\in\R,v\in\R^*\right\}$ is the
stabilizer subgroup in $\gl2$ of the point $(1,0)$
in $\widehat{\R^2}$. The calculation in \eqref{stabilize_1_0} shows that 
$\gamma(\uomega)^{-1}A\gamma(\uomega A)\in H_{(1,0)}$. When you calculate out $\gamma(\uomega)^{-1}A\gamma(\uomega A)$, the result is the matrix
$\begin{pmatrix} 
1 & 0\\
u_{\uomega,A} & v_{\uomega,A}
\end{pmatrix}$, where $u_{\uomega,A}$ and 
$v_{\uomega,A}$ are as in \eqref{uv}. The identity
in the following proposition was found in
\cite{Mil} and appears as Proposition 5.2(d) in
\cite{MT}. We repeat its brief proof here for
convenience.
\begin{prop}\label{uv_ident}
The functions defined in \eqref{uv} satisfy
\[
u_{\uomega,AB}=
u_{\uomega,A}+v_{\uomega,A}u_{\uomega A,B} 
\text{ and }
v_{\uomega,AB}=v_{\uomega,A}v_{\uomega A,B},
\]
for all $A,B\in\gl2$ and $\uomega\in \cO$.
\end{prop}
\begin{proof}
Let $A,B\in\gl2$ and $\uomega\in\cO$. Then
\[
\gamma(\uomega)^{-1}AB\gamma(\uomega AB)=
\big(\gamma(\uomega)^{-1}A\gamma(\uomega A)\big)
\big(\gamma(\uomega A)^{-1}B\gamma(\uomega AB)\big).
\]
That is
$\begin{pmatrix} 
1 & 0\\
u_{\uomega,AB} & v_{\uomega,AB}
\end{pmatrix}=
\begin{pmatrix} 
1 & 0\\
u_{\uomega,A} & v_{\uomega,A}
\end{pmatrix}
\begin{pmatrix} 
1 & 0\\
u_{\uomega A,B} & v_{\uomega A,B}
\end{pmatrix}$
and the two identities in the proposition follow by multiplying 
out the matrices on the right.
\end{proof}
We make use of the parametrization of $K_0$ by $\gamma$.
For $F\in L^2\big(K_0,L^2(\R^*)\big)$, 
$[\ux,A]\in G_2$, and a.e. $\uomega\in\cO$,
$\sigma[\ux,A]F(\gamma(\uomega))$ is the
element of $L^2(\R^*)$ given by
\begin{equation}\label{sigma_pointwise_formula_1}
\big(\sigma[\ux,A]F(\gamma(\uomega))\big)(\nu) =  
\textstyle\frac{|\det(A)|^{1/2}\|\uomega\|}{\|\uomega A\|}
e^{2\pi i(\uomega\ux+t^{-1}u_{\uomega,A})}
\big(F(\gamma(\uomega A))\big)
(v_{\uomega,A}^{-1}\nu),
\end{equation}
for a.e. $\nu\in\R^*$. It is shown in \cite{MT} that
$\sigma$ is an irreducible representation of $G_2$.
Moreover, $\sigma$ is square--integrable. In fact,
the following is Theorem 5.9 in \cite{MT} with
$L^2\big(K,L^2(\R^*)\big)$ replaced by
$L^2\big(K_0,L^2(\R^*)\big)$.
.
\begin{prop}\label{V_E_isometry}
Let $g\in L^2(\R^*)$ satisfy $\int_{\R^*}\big|
|\nu|^{1/2}g(\nu)\big|^2d\mu_{\R^*}(\nu)=1$ and let
$\eta\in L^2(\widehat{\R^2})$ satisfy
$\|\eta\|_{_{L^2(\widehat{\R^2})}}=1$.
Let $E\in  L^2\big(K_0,L^2(\R^*)\big)$ be defined
as 
\[
\big(E(L)\big)(\nu)=
\frac{\overline{\eta}((1,0)L^{-1})}{|\det(L)|}\,g(\nu),
\text{ for a.e. } L\in K_0 \text{ and }
\nu\in\R^*.
\]
 Define $V_EF[\ux,A]=
\langle F,
\sigma[\ux,A]E\rangle_{_{L^2(K,L^2(\R^*))}}$, for
$[\ux,A]\in G_2$ and 
$F \in L^2\big(K_0,L^2(\R^*)\big)$. Then $V_E$ is
a linear isometry of $L^2\big(K_0,L^2(\R^*)\big)$
into $L^2(G_2)$. In particular, $\sigma$ is a square-integrable representation of $G_2$.
\end{prop}

We can make $\sigma$ more concrete by 
using a unitary map
to move to the Hilbert space $L^2(\widehat{\R^3})$.
The formulas are simplified if we write 
$\widehat{\R^3}$ as 
$\widehat{\R^2}\times\widehat{\R}$. Define 
$U:L^2\big(K_0,L^2(\R^*))\big)\to 
L^2\big(\widehat{\R^2}\times\widehat{\R}\big)$ by,
for $F\in L^2\big(K_0,L^2(\R^*)\big)$ and a.e.
$(\uomega,\omega_3)\in 
\widehat{\R^2}\times\widehat{\R}$,
\begin{equation}\label{U_definition}
(UF)(\uomega,\omega_3)=\begin{cases}
\frac{\big(F(\gamma(\uomega))\big)(\omega_3^{-1})}{\|\uomega\|\cdot|\omega_3|^{1/2}} & \text{for }
\uomega\in\cO, \omega_3\neq 0\\
0 & \text{otherwise.}
\end{cases} 
\end{equation}
\begin{prop}\label{U_unitary}
The map $U:L^2\big(K_0,L^2(\R^*) \big)\to 
L^2\big(\widehat{\R^2}\times\widehat{\R}\big)$ is
unitary with inverse $U^{-1}$ given by, for 
$\xi\in L^2\big(\widehat{\R^2}\times\widehat{\R}\big)$
and a.e $L\in K_0$,
\[
\left(U^{-1}\xi(L)\right)(\nu)= 
\frac{\|(1,0)L^{-1}\|\,\xi\big((1,0)L^{-1},\nu^{-1}\big)}{|\nu|^{1/2}} \text{ for a.e. } \nu\in\R^*.
\]
\end{prop}
\begin{proof}
For $F \in L^2\big(K_0,L^2(\R^*)\big)$, 
\[
\begin{split}
\| UF \|^{2}_{L^{2}(\widehat{\R^{2}} \times \widehat{\R})}
&= \int_{\widehat{\R^{2}}} \int_{\widehat{\R}} 
|UF(\uomega,\omega_3)|^{2} d\omega_{3} \, d\uomega  \\
&=\int_{\widehat{\R^{2}}} \int_{\widehat{\R}}  \left| \frac{\big( F(\gamma( \uomega))\big)( \omega^{-1}_{3})}{\| \uomega \| \cdot | \omega_{3}|^{1/2}} \right|^{2} d\omega_{3} \, d\uomega   \\
&=  \int_{\widehat{\R^{2}}} \int_{\widehat{\R}} \left| 
\big( F(\gamma( \uomega))\big)( \omega^{-1}_{3}) \right|^{2} \frac{d\omega_{3}}{| \omega_{3} |} \, \frac{d\uomega }{ \| \uomega \|^{2} }  \\
&=  \int_{K_0} \int_{{\R^{*}}} \left| \big( F(L) \big)
( \omega^{-1}_{3}) \right|^{2} d\mu_{\R^{*}}(\omega_{3}) \, d\mu_{K_0}(L) \\
&= \|F \|_{L^{2}(K_0, L^{2}(\R^{*}))}
\end{split}
\]
This shows that $ U $ is an isometry. Verifying
that $U$ is onto and the formula for $U^{-1}$
can be done with similar calculations.
\end{proof}
For $[\ux,A]\in G_2$, let $\sigma_2[\ux,A]=  
U\sigma[\ux,A]U^{-1}$. Then $\sigma_2$ is an 
irreducible representation of $G_2$ on 
$L^2\big(\widehat{\R^2}\times\widehat{\R}\big)$. To
find a formula for $\sigma_2$, let $\xi\in 
L^2\big(\widehat{\R^2}\times\widehat{\R}\big)$ and
set $F=U^{-1}\xi\in L^2\big(K_0,L^2(\R^*)\big)$. 
Use formula \eqref{sigma_pointwise_formula_1} for
$\sigma$.
For
$[\ux,A]\in G_2$ and a.e. $(\uomega,\omega_3)\in 
\widehat{\R^2}\times\widehat{\R}$, 
\begin{equation*}
\begin{split}
\big(\sigma_2[\ux,A]\xi\big)&(\uomega,\omega_3)  =
\big(U\sigma[\ux,A]F\big)(\uomega,\omega_3)=
\frac{\big(\sigma[\ux,A]F(\gamma(\uomega))\big)(\omega_3^{-1})}{\|\uomega\|\cdot|\omega_3|^{1/2}}\\
& =\frac{|\det(A)|^{1/2}}
{\|\uomega A\|\cdot|\omega_3|^{1/2}}
e^{2\pi i(\uomega\ux+\omega_3u_{\uomega,A})}
\big(U^{-1}\xi[\u0,\gamma(\uomega A)]\big)
(v_{\uomega,A}^{-1}\omega_3^{-1})\\
& = |v_{\uomega,A}|^{1/2}|\det(A)|^{1/2}
e^{2\pi i(\uomega\ux+\omega_3u_{\uomega,A})}
\xi(\uomega A,\omega_3v_{\uomega,A}),
\end{split}
\end{equation*}
where we have used $(1,0)\gamma(\uomega A)^{-1}=
\uomega A$.
Finally by \eqref{uv},
$v_{\uomega,A}=\frac{\det(A)\|\uomega\|^2}{\|\uomega A\|^2}$. So
\begin{equation}\label{sigma_1_1}
\big(\sigma_2[\ux,A]\xi\big)(\uomega,\omega_3)=
\textstyle\frac{|\det(A)|\cdot\|\uomega\|}{\|\uomega A\|}
e^{2\pi i(\uomega\ux+\omega_3 u_{\uomega,A})}
\xi\big(\uomega A,\omega_3v_{\uomega,A}\big).
\end{equation}
We also need $UE$, where 
$E\in L^2\big(K_0,L^2(\R^*)\big)$ is the kind of 
function used in Proposition \ref{V_E_isometry}.  
\begin{lemma}\label{E_to_psi}
Let $\zeta\in L^2(\widehat{\R^2})$ satisfy the
condition
$\int_{\widehat{\R^2}}|\zeta(\uomega)|^2
\frac{d\uomega}{\|\uomega\|^2}=1$ and
let
$\phi\in L^2(\widehat{\R})$ satisfy
$\int_{\widehat{\R}}|\phi(\omega_3)|^2
\frac{d\omega_3}{|\omega_3|}=1$.

\noindent
{\rm (a)} If $\eta(\uomega)=
\overline{\zeta}(\uomega)/\|\uomega\|$, for a.e.
$\uomega\in\widehat{\R^2}$, then 
$\eta\in L^2(\widehat{\R^2})$ with
$\|\eta\|_{_{L^2(\widehat{\R^2})}}=1$.  

\noindent
{\rm (b)} If
$g(\nu)=\frac{1}{|\nu|^{1/2}}\phi(\nu^{-1})$,
for a.e. $\nu\in\R^*$. Then
$g\in L^2(\R^*)$ and $\int_{\R^*}|\nu|\cdot|
g(\nu)|^2d\mu_{\R^*}(\nu)=1$. 

\noindent
{\rm (c)} If $\big(E(L)\big)(\nu)=
\frac{\overline{\eta}((1,0)L^{-1})}{|\det(L)|}\,g(\nu)$,
for a.e. $L\in K_0$, $\nu\in\R^*$, then 
$E\in L^2\big(K_0,L^2(\R^*)\big)$ and
\[
(UE)(\uomega,\omega_3)=\zeta(\uomega)\phi(\omega_3),
\text{ for a.e. }(\uomega,\omega_3)\in 
\widehat{\R^2}\times\widehat{\R^*}.
\]
\end{lemma}
\begin{proof}
(a) and (b) are established by simple calculations. For
(c), recall that $(1,0)\gamma(\uomega)^{-1})=\uomega$ 
and $\det\big(\gamma(\uomega)
\big)=\|\uomega\|^{-2}$. So
\[
(UE)(\uomega,\omega_3)=
\frac{\overline{\eta}((1,0)\gamma(\uomega)^{-1})}{
\|\uomega\|\cdot|\omega_3|^{1/2}|\det(\gamma(\uomega))|}\,g(\omega_3^{-1})=
\|\uomega\|\overline{\eta}(\uomega)|\omega_3|^{-1/2}
g(\omega_3^{-1})=\zeta(\uomega)\phi(\omega_3),
\]
for a.e. $(\uomega,\omega_3)\in 
\widehat{\R^2}\times\widehat{\R}$.
\end{proof}
The results in Lemma \ref{E_to_psi} enable the following restatement of Proposition \ref{V_E_isometry}.

\begin{prop}\label{sigma_2_prop}
For each $[\ux,A]\in G_2$, define $\sigma_2[\ux,A]$ on
$L^2\big(\widehat{\R^2}\times\widehat{\R}\big)$ by,
for $\xi\in 
L^2\big(\widehat{\R^2}\times\widehat{\R}\big)$
\[
\big(\sigma_2[\ux,A]\xi\big)(\uomega,\omega_3)=
\textstyle\frac{|\det(A)|\cdot\|\uomega\|}{\|\uomega A\|}
e^{2\pi i(\uomega\ux+\omega_3 u_{\uomega,A})}
\xi\big(\uomega A,\omega_3v_{\uomega,A}\big),
\]
for a.e. $(\uomega,\omega_3)\in 
\widehat{\R^2}\times\widehat{\R}$. Then $\sigma_2$
is an irreducible representation of $G_2$ that is
equivalent to $\sigma$ as given in \eqref{sigma_pointwise_formula_1}. 
Let $\zeta\in L^2(\widehat{\R^2})$ satisfy 
$\int_{\widehat{\R^2}}|\zeta(\uomega)|^2
\frac{d\uomega}{\|\uomega\|^2}=1$ and
let
$\phi\in L^2(\widehat{\R})$ satisfy
$\int_{\widehat{\R}}|\phi(\omega_3)|^2
\frac{d\omega_3}{|\omega_3|}=1$. Define
$\psi(\uomega,\omega_3)=\zeta(\uomega)\phi(\omega_3)$,
for a.e. $(\uomega,\omega_3)\in 
\widehat{\R^2}\times\widehat{\R}$. Then $\psi\in 
L^2\big(\widehat{\R^2}\times\widehat{\R}\big)$.
For $\xi\in L^2\big(\widehat{\R^2}\times\widehat{\R}\big)$,
let $V_\psi\xi[\ux,A]=\langle\xi,\sigma_2[\ux,A]\psi
\rangle_{L^2(\widehat{\R^2}\times\widehat{\R})}$, for
all $[\ux,A]\in G_2$. Then $V_\psi\xi\in L^2(G_2)$
and $\|V_\psi\xi\|_{L^2(G_2)}=
\|\xi\|_{L^2(\widehat{\R^2}\times\widehat{\R})}$,
for each $\xi\in 
L^2\big(\widehat{\R^2}\times\widehat{\R}\big)$. In 
particular, $\sigma_2$ is square--integrable.
\end{prop}
Notice that $\psi$ as defined in Proposition \ref{sigma_2_prop} satisfies 
$\int_{\widehat{\R}}\int_{\widehat{\R^2}}
\frac{|\psi(\uomega,\omega_3)|^2}{\|\uomega\|^2|\omega_3|}
d\uomega\,d\omega_3=1$. This fact, together with
the fact that $V_\psi$ is an isometry, suggests a
form for $C_{\sigma_2}$, the Duflo--Moore operator
whose existence is part of Theorem \ref{Du_Mo}.
Our candidate is an operator $T$ whose domain is
\[
\cD_T
= 
\left\{
\xi\in L^2\big(\widehat{\R^2}\times\widehat{\R}\big):
\int_{\widehat{\R}}\int_{\widehat{\R^2}}
|\xi(\uomega,\omega_3)|^2\textstyle
\frac{d\uomega\,d\omega_3}{\|\uomega\|^2|\omega_3|}
<\infty\right\}.
\]
Then $\cD_T$ is a dense subspace of 
$L^2\big(\widehat{\R^2}\times\widehat{\R}\big)$.
We define a {\em multiplication operator} $T$ on
$L^2\big(\widehat{\R^2}\times\widehat{\R}\big)$
by, for $\xi\in 
L^2\big(\widehat{\R^2}\times\widehat{\R}\big)$,
\[
T\xi(\uomega,\omega_3)=\|\uomega\|^{-1} |\omega_3|^{-1/2}
\xi(\uomega,\omega_3), \text{ for a.e. }
(\uomega,\omega_3)\in \widehat{\R^2}\times\widehat{\R}.
\]
Then $T\xi$ is a Borel measurable function on
$\widehat{\R^2}\times\widehat{\R}$, for any
$f\in 
L^2\big(\widehat{\R^2}\times\widehat{\R}\big)$.
However
$T\xi\in 
L^2\big(\widehat{\R^2}\times\widehat{\R}\big)$
if and only if $\xi\in\cD_T$. Thus, $T$ is an
unbounded operator. For any $\xi\in\cD_T$,
$\langle T\xi,\xi\rangle_{_{L^2(\widehat{\R^2}\times\widehat{\R})}}\geq 0$. That is, $T$ is a 
{\em positive} operator.
\begin{lemma}\label{D_T_invariant}
For any $\xi\in\cD_T$, $\sigma_2[\ux,A]\xi\in\cD_T$,
for all $[\ux,A]\in G_2$.
\end{lemma}
\begin{proof}
Fix $[\ux,A]\in G_2$. For any $\xi\in\cD_T$ and
a.e. $(\uomega,\omega_3)\in \widehat{\R^2}\times
\widehat{\R}$, by \eqref{sigma_1_1}
\[
\begin{split}
\textstyle
\int_{\widehat{\R}}\int_{\widehat{\R^2}}
|\sigma_2[\ux,A]\xi(\uomega,\omega_3)|^2 
\frac{d\uomega\,d\omega_3}{\|\uomega\|^2|\omega_3|}
& =\textstyle
|\det(A)|^2\int_{\widehat{\R}}\int_{\widehat{\R^2}}
\frac{\|\uomega\|^2}{\|\uomega A\|^2}|\xi(\uomega A,\omega_3v_{\uomega,A})|^2
\frac{d\uomega\,d\omega_3}{\|\uomega\|^2|\omega_3|}\\
& = |\det(A)|\int_{\widehat{\R}}\int_{\widehat{\R^2}}
\textstyle\frac{1}{\|\uomega\|^2}|\xi(\uomega,\omega_3v_{\uomega A^{-1},A})|^2
\frac{d\uomega\,d\omega_3}{|\omega_3|}\\
& = 
|\det(A)|\int_{\widehat{\R^2}}\int_{\widehat{\R}}
\textstyle\frac{1}{\|\uomega\|^2}|\xi(\uomega,\omega_3v_{\uomega A^{-1},A})|^2
\frac{d\omega_3}{|\omega_3|}d\uomega\\
& = |\det(A)|\int_{\widehat{\R}}\int_{\widehat{\R^2}}|\xi(\uomega,\omega_3)|^2
\frac{d\uomega\,d\omega_3}{\|\uomega\|^2|\omega_3|}
<\infty.
\end{split}
\]
Thus, $\sigma_2[\ux,A]\xi\in\cD_T$.
\end{proof}

So $\cD_T$ is a (non-closed) $\sigma_2$-invariant
subspace of 
$L^2\big(\widehat{\R^2}\times\widehat{\R}\big)$.
For any $[\ux,A]\in G_2$, it is now easy to see that 
$\sigma_2[\ux,A]T\sigma_2[\ux,A]^*$ is a 
self-adjoint positive operator with domain $\cD_T$.
It is an important fact that 
$\sigma_2[\ux,A]T\sigma_2[\ux,A]^*$ is just a
special multiple of the operator $T$. We recall that
the modular function of $G_2$ is given by
$\Delta_{G_2}[\ux,A]=|\det(A)|^{-1}$, for $[\ux,A]
\in G_2$.
\begin{prop}\label{semi_invariance}
For any $[\ux,A]\in G_2$, $\sigma_2[\ux,A]T
\sigma_2[\ux,A]^*=\Delta_{G_2}[\ux,A]^{1/2}T$.
\end{prop}
\begin{proof} 
Fix $[\ux,A]\in G_2$.
For any $\xi\in \cD_T$, let 
$\xi'=\sigma_2[\ux,A]^*\xi=
\sigma_2[-A^{-1}\ux,A^{-1}]\xi$. 
Using \eqref{sigma_1_1},
\begin{equation}\label{semi_1}
\begin{split}
\sigma_2[\ux,A]T
\xi'(\uomega,\omega_3) & =
\textstyle\frac{|\det(A)|\cdot\|\uomega\|}{\|\uomega A\|}
e^{2\pi i(\uomega\ux+\omega_3 u_{\uomega,A})}
\big(T\xi'\big)
\big(\uomega A,\omega_3v_{\uomega,A})\\
& =
\textstyle\frac{|\det(A)|\cdot\|\uomega\|}{\|\uomega A\|^2|\omega_3v_{\uomega,A}|^{1/2}}
e^{2\pi i(\uomega\ux+\omega_3 u_{\uomega,A})}
\xi'
\big(\uomega A,\omega_3v_{\uomega,A}),
\end{split}
\end{equation}
for a.e. $(\uomega,\omega_3)\in 
L^2\big(\widehat{\R^2}\times\widehat{\R}\big)$.
On the other hand, 
\begin{equation}\label{semi_2}
\begin{split}
\xi'
\big(\uomega A,\omega_3v_{\uomega,A}) & =
\sigma_2[-A^{-1}\ux,A^{-1}]\xi
\big(\uomega A,\omega_3v_{\uomega,A})\\
& = \textstyle\frac{|\det(A^{-1}|\cdot \|\uomega A\|}{\|\uomega AA^{-1}\|}
e^{2\pi i(\uomega A(-A^{-1}\ux)+\omega_3v_{\uomega,A}u_{\uomega A,A^{-1}})}
\xi(\uomega,\omega_3)\\
& = \textstyle\frac{|\det(A|^{-1} \|\uomega A\|}{\|\uomega\|}
e^{-2\pi i(\uomega\ux+\omega_3u_{\uomega,A})}
\xi(\uomega,\omega_3)
\end{split}
\end{equation}
using $v_{\uomega,A}v_{\uomega A,A^{-1}}=1$
and $v_{\uomega,A}u_{\uomega A,A^{-1}}=
-u_{\uomega,A}$, which follow from Proposition
\ref{uv_ident} and the observation
that $u_{\uomega,{\rm id}}=0$ and
$v_{\uomega,{\rm id}}=1$, for any
$\uomega\in\cO$. Inserting \eqref{semi_2} in
\eqref{semi_1}, using $|v_{\uomega,A}|^{1/2}= 
\frac{|\det(A)|^{1/2}\|\uomega\|}{\|\uomega A\|}$,
and reducing gives
\[
\begin{split}
\sigma_2[\ux,A]T\sigma_2[\ux,A]^*
\xi(\uomega,\omega_3) & = 
|\det(A)|^{-1/2}\|\uomega\|^{-1}|\omega_3|^{-1/2}
\xi(\uomega,\omega_3)\\
& =
\Delta_G[\ux,A]^{1/2}T\xi(\uomega,\omega_3),
\end{split}
\]
for a.e. $(\uomega,\omega_3)\in 
\widehat{\R^2}\times\widehat{\R}$.
This shows that 
$\sigma_2[\ux,A]T\sigma_2[\ux,A]^*$ and 
$\Delta_G[\ux,A]T$ agree as self-adjoint positive
operators on $L^2\big(\widehat{\R^2}\times\widehat{\R}\big)$.
\end{proof}

In the terminology of \cite{DM}, Proposition 
\ref{semi_invariance} says $T$ is 
semi-invariant, with weight $\Delta_G^{1/2}$,
with respect to the irreducible representation
$\sigma_2$. However, Proposition \ref{sigma_2_prop}
says that $\sigma_2$ is a square-integrable
representation. Let $C_{\sigma_2}$ denote the
Duflo-Moore operator associated with $\sigma_2$
as described in Theorem \ref{Du_Mo}. In
particular, $C_{\sigma_2}$ is 
semi-invariant, with weight $\Delta_G^{1/2}$,
with respect to the irreducible representation
$\sigma_2$ as well. By the uniqueness part of
Theorem \ref{Du_Mo}, there
is a positive constant $r$ such that 
$C_{\sigma_2}=rT$.
\begin{theorem}\label{Duflo_Moore_operator}
The Duflo-Moore operator 
$C_{\sigma_2}$ associated with
$\sigma_2$ is given by, for any $\xi\in 
L^2\big(\widehat{\R^2}\times\widehat{\R}\big)$,
$C_{\sigma_2}\xi(\uomega,\omega_3)=
\|\uomega\|^{-1}|\omega_3|^{-1/2}
\xi(\uomega,\omega_3)$, for a.e.
$(\uomega,\omega_3)\in 
\widehat{\R^2}\times\widehat{\R}$.
\end{theorem}
\begin{proof}
Let $r>0$ be such that $C_{\sigma_2}=rT$.
Let $\zeta\in L^2\big(\widehat{\R^2}\big)$ 
satisfy the condition that
$\int_{\widehat{\R^2}}
|\zeta(\uomega)|^2\frac{d\uomega}{\|\uomega\|^2}=1$
and $\phi\in L^2\big(\widehat{\R}\big)$ 
satisfy $\int_{\widehat{\R}}|\phi(\omega_3)|^2
\frac{d\omega_3}{|\omega_3|}=1$. Define $\psi$ from
$\zeta$ and $\phi$ as in Proposition 
\ref{sigma_2_prop}. Then 
\[
\begin{split}
\|T\psi\|_{L^2(\widehat{\R^2}\times\widehat{\R})}^2 & =
\int_{\widehat{\R}}\int_{\widehat{\R^2}}
\big|\|\uomega\|^{-1}|\omega_3|^{-1/2}
\psi(\uomega,\omega_3)\big|^2d\uomega\,d\omega_3\\
& = \int_{\widehat{\R}}\left(\int_{\widehat{\R^2}}
|\zeta(\uomega)|^2\textstyle\frac{d\uomega}{\|\uomega\|^2}\right)|\phi(\omega_3)|^2
\textstyle\frac{d\omega_3}{|\omega_3|} = 1.
\end{split}
\]
Moreover, $V_\psi$ is an isometry of 
$L^2\big(\widehat{\R^2}\times\widehat{\R}\big)$
into $L^2(G_2)$. 
On the other hand, from Theorem 2.25 of \cite{Fuh}
we see that $V_\psi$ is an isometry of 
$L^2\big(\widehat{\R^2}\times\widehat{\R}\big)$
into $L^2(G_2)$ exactly when 
\[
1=\|C_{\sigma_2}\psi\|_{L^2(\widehat{\R^2}\times\widehat{\R})}=
\|rT\psi\|_{L^2(\widehat{\R^2}\times\widehat{\R})}
=r\|T\psi\|_{L^2(\widehat{\R^2}\times\widehat{\R})}
=r.
\]
Therefore, $C_{\sigma_2}=T$.
\end{proof}
The Duflo-Moore orthogonality relations 
(Theorem \ref{Du_Mo}(d)) for
the square-integrable representation $\sigma_2$
can now be stated.
\begin{cor}\label{ortho_relations_sigma_1}
Let $\xi_1,\xi_2\in 
L^2\big(\widehat{\R^2}\times\widehat{\R}\big)$ and
$\psi_1,\psi_2\in\cD_T$. Then
\[
\left\langle V_{\psi_1}\xi_1,V_{\psi_2}\xi_2
\right\rangle_{_{L^2(G_2)}}=
\langle\xi_1,\xi_2\rangle_{_{L^2(\widehat{\R^2}\times\widehat{\R})}}
\langle T\psi_2,T\psi_1\rangle_{_{L^2(\widehat{\R^2}\times\widehat{\R})}}.
\]
\end{cor}
Now, we no longer need $\psi$ to be formed from
functions $\zeta$ and $\phi$ as in Proposition 
\ref{sigma_2_prop}. We only need $\psi\in 
L^2\big(\widehat{\R^2}\times\widehat{\R}\big)$ and
$\|T\psi\|_{L^2(\widehat{\R^2}\times\widehat{\R})}
=1$. Then the identity in Corollary 
\ref{ortho_relations_sigma_1} implies the following
corollary.
\begin{cor}\label{sigma_1_wavelet2}
Let $\psi\in 
L^2\big(\widehat{\R^2}\times\widehat{\R}\big)$
satisfy
\begin{equation}\label{sigma_1_wavelet}
\int_{\widehat{\R}}\int_{\widehat{\R^2}}\frac{|\psi(\uomega,\omega_3)|^2}{\|\uomega\|^2|\omega_3|}d\uomega\,d\omega_3=1.
\end{equation}
Then $V_\psi:
L^2\big(\widehat{\R^2}\times\widehat{\R}\big)
\to L^2(G_2)$ is a linear isometry.
\end{cor}

\section{The Generalized Wavelet Transforms}
Corollaries \ref{ortho_relations_sigma_1} and
\ref{sigma_1_wavelet2} are the critical 
parts of forming a continuous wavelet transform
and reconstruction formula arising from the
square-integrable representation $\sigma_2$ of
$G$. In this section, we bring together the various 
details of the transform as well as the version where
an inverse Fourier transform in the third dimension
is carried.
\begin{defi}
A function $\psi\in 
L^2\big(\widehat{\R^2}\times\widehat{\R}\big)$ is
called a {\em $\sigma_2$-wavelet} if
\[
\int_{\widehat{\R^2}}\int_{\widehat{\R}}
\frac{|\psi(\uomega,\omega_3)|^2}{\|\uomega\|^2|\omega_3|}d\uomega\,d\omega_3=1.
\]
\end{defi}
 For each
$\ux\in\R^2$ and $A=\begin{pmatrix}
a & b\\
c & d
\end{pmatrix}\in\gl2$,
define
$\psi_{\ux,A}$ on 
$\widehat{\R^2}\times\widehat{\R}$ by
\[
\psi_{\ux,A}(\uomega,\omega_3)=
\textstyle\frac{|\det(A)|\cdot\|\uomega\|}{\|\uomega A\|}
e^{2\pi i(\uomega\ux+\omega_3 u_{\uomega,A})}
\psi\big(\uomega A,\omega_3v_{\uomega,A}),
\text{ for a.e. } (\uomega,\omega_3)\in 
\widehat{\R^2}\times\widehat{\R},
\]
where 
\[
u_{\uomega,A}=\frac{(ac+bd)(\omega_1^2-\omega_2^2)-(a^2+b^2-c^2-d^2)\omega_1\omega_2}{(a\omega_1+c\omega_2)^2+(b\omega_1+d\omega_2)^2}
\]
and
\[
v_{\uomega,A}=\frac{(ad-bc)(\omega_1^2+\omega_2^2)}{(a\omega_1+c\omega_2)^2+(b\omega_1+d\omega_2)^2}.
\]
Then $\psi_{\ux,A}\in 
L^2\big(\widehat{\R^2}\times\widehat{\R}\big)$.
\begin{defi}
For each $\xi\in 
L^2\big(\widehat{\R^2}\times\widehat{\R}\big)$, let
\[
V_\psi\xi[\ux,A]=\left\langle\xi,\psi_{\ux,A}
\right\rangle_{L^2(\widehat{\R^2}\times\widehat{\R})}, \text{ for all }\ux\in\R^2,
A\in\gl2.
\]
Then $V_\psi$ is called the {\em $\sigma_2$-wavelet
transform} with $\sigma_2$-wavelet $\psi$.
\end{defi}
The reconstruction formula given in 
\eqref{reconstruction} can now be stated for the
$\sigma_2$-wavelet transform.
\begin{theorem}\label{reconstruction_final1}
Let $\psi\in 
L^2\big(\widehat{\R^2}\times\widehat{\R}\big)$
be a $\sigma_2$-wavelet. Then, for any 
$\xi\in 
L^2\big(\widehat{\R^2}\times\widehat{\R}\big)$,
\[
\xi =\int_{\gl2}\int_{\R^2}V_\psi\xi[\ux,A]\,
\psi_{\ux,A}\,\textstyle\frac{d\ux\,d\mu_{\gl2}(A)}{|\det(A)|},
\]
weakly in $L^2\big(\widehat{\R^2}\times\widehat{\R}\big)$.
\end{theorem}
Recall that 
the integral in Theorem
\ref{reconstruction_final1} is over the 6 real
variables $x_1,x_2,a,b,c,d$, where 
$A=\begin{pmatrix}
a & b\\c & d
\end{pmatrix}$ and 
$\frac{d\ux\,d\mu_{\gl2}(A)}{|\det(A)|}=
\frac{dx_1\,dx_2\,da\,db\,dc\,dd}{|ad-bc|^3}$.

This transform can also be expressed in a related manner by taking an inverse Fourier transform in the 
distinguished third variable.

Let $\cF_3: L^2\big(\widehat{\R^2}\times\R\big)\to
L^2\big(\widehat{\R^2}\times\widehat{\R}\big)$ be the
unitary map such that
\[
\cF_3f(\uomega,\omega_3)=
\int_{\R}f(\uomega,t)e^{2\pi i\omega_3t}dt,
\text{ for all }(\uomega,t)\in \widehat{\R^2}\times\R
\]
and $f\in C_c\big(\widehat{\R^2}\times\R\big)$. 
Then
$\cF_3^{-1}: L^2\big(\widehat{\R^2}\times\R\big)\to
L^2\big(\widehat{\R^2}\times\widehat{\R}\big)$ is
such that
\[
\cF_3^{-1}\xi(\uomega,t)=
\int_{\widehat{\R}}\xi(\uomega,\omega_3)
e^{-2\pi i\omega_3t}d\omega_3,
\text{ for any } (\uomega,t)\in 
\widehat{\R^2}\times\R,
\]
and $\xi\in C_c\big(\widehat{\R^2}\times\widehat{\R}
\big)$. Use $\cF_3$ to move $\sigma_2$ to an
equivalent representation acting on 
$L^2\big(\widehat{\R^2}\times\R\big)$.

\begin{defi}
For each $[\ux,A]\in G_2$, let
$\rho_2[\ux,A]=\cF_3^{-1}\sigma_2[\ux,A]\cF_3$.
\end{defi}
\begin{prop}
The map
$\rho_2$ is a square-integrable representation
of the affine group 
$G_2$ on the Hilbert space
$L^2\big(\widehat{\R^2}\times\R\big)$. For 
$f\in L^2\big(\widehat{\R^2}\times\R\big)$ and
$[\ux,A]\in G_2$,
\[
\rho_2[\ux,A]f(\uomega,t)=
\textstyle\frac{\|\uomega A\|}{\|\uomega\|}
e^{2\pi i\uomega\ux}f\left(\uomega A,
\frac{t-u_{\uomega,A}}{v_{\uomega,A}}\right),
\]
for a.e. $(\uomega,t)\in \widehat{\R^2}\times\R$.
\end{prop}
\begin{proof}
Since $\rho_2$ is equivalent to $\sigma_2$ (and 
$\sigma$), it is an irreducible representation
that is square--integrable. To verify its 
formula, fix $[\ux,A]\in G_2$ and
let $\xi=\cF_3f$. Then
\begin{equation}\label{rho_1}
\begin{split}
\rho_2[\ux,A]f(\uomega,t) & = 
\cF_3^{-1}\sigma_2[\ux,A]\xi(\uomega,t)=
\int_{\widehat{\R}}
\sigma_2[\ux,A]\xi(\uomega,\omega_3)
e^{-2\pi i\omega_3t}d\omega_3\\
& = \int_{\widehat{\R}}\textstyle
\frac{|\det(A)|\cdot\|\uomega\|}{\|\uomega A\|}
e^{2\pi i(\uomega\ux+\omega_3u_{\uomega,A})}
\xi(\uomega A,\omega_3v_{\uomega,A})
e^{-2\pi i\omega_3t}d\omega_3\\
& = {\textstyle
\frac{|\det(A)|\cdot\|\uomega\|}{\|\uomega A\|}}
e^{2\pi i(\uomega\ux}\int_{\widehat{\R}}
\xi(\uomega A,\omega_3v_{\uomega,A})
e^{-2\pi i\omega_3(t-u_{\uomega,A})}d\omega_3.
\end{split}
\end{equation}
Make the change of variables $\omega_3'=
\omega_3v_{\uomega,A}$. So $d\omega_3'=
|v_{\uomega,A}|d\omega_3=
\frac{|\det(A)|\cdot\|\uomega\|^2}{\|\uomega A\|^2}
d\omega_3$. Thus, \eqref{rho_1} becomes
\[
\begin{split}
\rho_2[\ux,A]f(\uomega,t) & =
{\textstyle\frac{\|\uomega A\|}{\|\uomega\|}}
e^{2\pi i\uomega\ux}\int_{\widehat{\R}}
\xi(\uomega A,\omega_3')
e^{-2\pi i\omega_3'v_{\uomega,A}^{-1}
(t-u_{\uomega,A})}d\omega_3'\\
& =
\textstyle\frac{\|\uomega A\|}{\|\uomega\|}
e^{2\pi i\uomega\ux}f\left(\uomega A,
\frac{t-u_{\uomega,A}}{v_{\uomega,A}}\right),
\end{split}
\]
for a.e. $(\uomega,t)\in \widehat{\R^2}\times\R$.
\end{proof}
\begin{defi}
A function $w\in L^2\big(\widehat{\R^2}\times\R\big)$ is called a {\em $\rho_2$-wavelet} if 
$\cF_3w$ is a $\sigma_2$-wavelet; that is
\[
\int_{\widehat{\R^2}}\int_{\widehat{\R}}
\frac{|\cF_3w(\uomega,\omega_3)|^2}{\|\uomega\|^2|\omega_3|}d\uomega\,d\omega_3=1.
\]
\end{defi}
If $w$ is a $\rho_2$-wavelet, for each $[\ux,A]\in G_2$,
define $w_{\ux,A}=\rho[\ux,A]w$. That is,
\[
w_{\ux,A}(\uomega,t)=
\textstyle\frac{\|\uomega A\|}{\|\uomega\|}
e^{2\pi i\uomega\ux}w\left(\uomega A,
\frac{t-u_{\uomega,A}}{v_{\uomega,A}}\right),
\text{ for a.e. }(\uomega,t)\in 
\widehat{\R^2}\times\R,
\]
where $u_{\uomega,A}$ and $v_{\uomega,A}$ are
as above.
\begin{defi}
For each $f\in L^2\big(\widehat{\R^2}\times\R\big)$,
let
\[
V_wf[\ux,A]=\langle f,w_{\uomega,A}
\rangle_{L^2(\widehat{\R^2}\times\R)}, 
\text{ for all }\ux\in\R^2, A\in \gl2.
\]
Then $V_w$ is the {\em $\rho_2$-wavelet transform}
with $\rho_2$-wavelet $w$.
\end{defi}
\begin{theorem}\label{reconstruction_final2}
Let $w\in 
L^2\big(\widehat{\R^2}\times\R\big)$
be a $\rho_2$-wavelet. Then, for any 
$f\in 
L^2\big(\widehat{\R^2}\times\R\big)$,
\[
f =\int_{\gl2}\int_{\R^2}V_wf[\ux,A]\,
w_{\ux,A}\,\textstyle\frac{d\ux\,d\mu_{\gl2}(A)}{|\det(A)|},
\]
weakly in $L^2\big(\widehat{\R^2}\times\R\big)$.
\end{theorem}


\begin{thebibliography}{100}
\bibitem{ACDL} P. Aniello, G. Cassinelli, E. de 
Vitto, and A. Levrero: Square-integrability of 
induced representations of semidirect products, 
{\em Rev. Math. Phys.}, {\bf 10}, 301--313, 1998.

\bibitem{ACMP} J.--P. Antoine, P. Carrette, R.
Murenzi, and B. Piette: Image Analysis with 
two--dimensional continuous wavelet transform, 
{\em Signal Processing}, {\bf 31}, 241--272, 1993.

\bibitem{AM} J.--P. Antoine and R. Murenzi: 
Two--dimensional directional wavelets and the 
scale--angle representation, {\em Signal Processing}, 
{\bf 52}, 259--281, 1996.

\bibitem{BagT} L. Baggett and K. F. Taylor: Groups with completely reducible regular representation, {\em Proc. AMS}, {\bf 72}, 593--600, 1978.

\bibitem{BT} D. Bernier and K.F. Taylor: {\em Wavelets from square-integrable representations.} {\it SIAM J. Math. Anal}., {\bf 27} 594--608, 1996.

\bibitem{CFO} B. Currey, H. F\"uhr, and V. Oussa: A
classification of continuous wavelet transforms in
dimension three, {\em Appl. and Comp. Har. Anal.},
{\bf 46}, 500--543, 2019.

\bibitem{DK} S. Dahlke, G. Kutyniok, P. Maass, C. Sagiv, H.-G. Stark, and G. Teschke.
{\em The Uncertainty Principle associated with the Continuous Shearlet Transform.}
   {\it Int. J . Wavelet Multiresolut. Inf. Process} 6( 2008), 157--181.
   
\bibitem{Da} I. Daubechies: {\em Ten lectures on Wavelets}. CBMS-NSF Regional Conference Series in Applied Mathematics, SIAM, 1992.

\bibitem{Dix} J. Dixmier: {\em C*-algebras} 
North-Holland, 1977.

\bibitem{DM} M. Duflo and C.C. Moore: On the 
Regular Representation of a Nonunimodular Locally
Compact Group, {\em J. of Functional Analysis} 
{\bf 21}, 209--243, 1976

\bibitem{Fol} G.B. Folland: {\em A Course in Abstract Harmonic Analysis}, CRC Press, 1995.

\bibitem{Fuh2}  H. F\"uhr: Wavelet frames and admissibility in higher dimensions,  {\em J. Math. Phys.} {\bf 37}, 6353--6366, 1996.

\bibitem{Fuh3}  H. F\"uhr: Continuous Wavelets
Transforms from Semidirect Products, {\em 
Revista Ciencias Matematicas} {\bf 18}, 179--190,
2000.

\bibitem{Fuh}  H. F\"uhr: {\em  Abstract harmonic analysis of continuous wavelet transforms.} Lecture Notes in Mathematics, 1863, Berlin: Springer, 2005.

\bibitem{GW} A. Gorodnik and B. Weiss: Distribution of Lattice Orbits on Homogeneous Varieties, {\em 
Geom. Funct. Anal.} {\bf 17}, 58--115, 2007.

\bibitem{GMPI} A. Grossmann, J. Morlet, and T. Paul:
Decomposition of Hardy functions into square integrable wavelets of constant shape, {\em SIAM
J Math Anal} {\bf 15}, 723--736, 1984.

\bibitem{GMP} A. Grossmann, J. Morlet, and T. Paul:
Transforms Associated to Square Integrable Group
Representations I: General Results, {\em J. Math.
Phys.}, {\bf 27}, 2473--2479, 1985.

\bibitem{GMPII} A. Grossmann, J. Morlet, and T. Paul:
Transforms Associated to Square Integrable Group
Representations II: Examples, {\em 
Annales de l'I. H. P., section A}, {\bf 45}, 
293--309, 1986.

\bibitem{GLL} K. Guo, D. Labate, and W.--Q. Lim:
Edge analysis and identification using the 
continuous shearlet transform, {\em Appl. Comput. 
Harmon. Anal.},  {\bf 27}, 24--46, 2009.

\bibitem{HW} C.E. Heil and D.F. Walnut: Continuous and discrete wavelet transforms, {\em SIAM Review} {\bf 31}, 628--666, 1989.

\bibitem{HR}  E. Hewitt and K.A. Ross:  {\em Abstract harmonic analysis. I}, Berlin: Springer, 1963.

 \bibitem{KR} Richard V. Kadison and John R. Ringrose: {\em Fundamentals of the Theory of Operator Algebras }. vol. I, Academic press, San Diego, California, 1983.

\bibitem{KT} E. Kaniuth and K.F. Taylor: {\em Induced representations of locally compact groups}, Cambridge Tracts in
Mathematics, Vol. 197, Cambridge University Press, 2013.

\bibitem{KL} A. Kleppner and R. Lipsman: The 
Plancherel Formula for Group Extensions, {\em 
Ann. Sci. de l'\'E.N.S. $4^e$ s\'erie}, {\bf 5},
459--516, 1972.

\bibitem{KuLa} G. Kutyniok and D. Labate: 
Resolution of the wavefront set using continuous
shearlets, {\em Trans. AMS}, {\bf 361}, 2719--2754, 
2009.
 
 \bibitem{Mac} G.W. Mackey: Induced 
Representations of Locally Compact Groups I, 
{\em Ann. of Math.} {\bf 55}, 101--139, 1952.

\bibitem{Mil} R. Milad: {\em Harmonic Analysis On Affine Groups: Generalized Continuous Wavelet Transforms}, Doctoral Thesis, Dalhousie University,
2021.
http://hdl.handle.net/10222/80691

\bibitem{MT} R. Milad and K.F. Taylor: 
Harmonic Analysis on the Affine Group of 
the Plane ({\em Preprint}).

\bibitem{Muz} R. Murenzi: {\em Ondelettes 
multidimensionelles et application \`a l'analyse 
d'images}. Th\`ese, Universit\'e Catholique de Louvain,
Louvain--La--Neuve, 1990.

\bibitem{Ru} W. Rudin: {\em Real and Complex
Analysis}, McGraw-Hill, 1987.

\end{thebibliography}
\end{document}